\documentclass{amsart}

 \usepackage{doc}
 \usepackage{amssymb}
 \usepackage{amsmath}
 \usepackage{amsthm}
 \usepackage{lipsum}

 \usepackage{graphicx}
\usepackage{amsaddr}
\makeatletter
\renewcommand{\email}[2][]{%
  \ifx\emails\@empty\relax\else{\g@addto@macro\emails{,\space}}\fi%
  \@ifnotempty{#1}{\g@addto@macro\emails{\textrm{(#1)}\space}}%
  \g@addto@macro\emails{#2}%
}
\makeatother

\newtheorem{theorem}{Theorem}
\newtheorem{lemma}{Lemma}
\newtheorem{corollary}{Corollary}
\newtheorem{proposition}{Proposition}

\begin{document}

\title{Special solitons on $3$-manifolds}

\author{Nasrin Malekzadeh$^{ 1}$ and  Esmaiel Abedi$^{2}$}
\address{$^{1}$, $^{2}$ Department of Mathematics, Azarbaijan Shahid Madani University, Tabriz 53751 71379, I. R. Iran } 
\email{n.malekzadeh@azaruniv.edu (Corresponding author)}

\email{esabedi@azaruniv.ac.ir}



\maketitle

\begin{abstract}
In this paper we study solitons on $3$-dimensional manifolds. In particular, we show that $3$-dimensional pseudo-symmetric gradient Ricci solitons and nontrivial gradient Yamabe solitons are locally isometric to either $\mathbb{R}^3$, $\mathbb{S}^3$, $\mathbb{H}^3$, $\mathbb{R}\times\mathbb{S}^2$ or $\mathbb{R}\times\mathbb{H}^2$. 
\end{abstract} 



\section{Introduction}
The Ricci flow was introduced by Hamilton \cite{43} to study compact three-manifolds with positive Ricci curvature. The Ricci solitons, as a generalization of the Einstein metrics, are self-similar solutions of Hamilton's Ricci flow. 

A Riemannian manifold $(M, g)$ is said to be a Ricci soliton if there exist $\lambda\in \mathbb{R}$ and $V\in \mathfrak{X}(M)$ such that 
\begin{equation}\label{020}
\dfrac{1}{2}\mathcal{L}_V g+Ric=\lambda g.
\end{equation}
The Ricci soliton is said to be expanding, steady or shrinking according as $\lambda > 0$, $\lambda=0$ or $\lambda< 0$, respectively. If the vector field $V$ is the gradient of a smooth function $f$, i.e., $V=\nabla f$ then $g$ is called a gradient Ricci soliton and the function $f$ is called the potential function. In this case equation \eqref{020} reduces to
\begin{equation}\label{021} 
Hess f+Ric=\lambda g,       
\end{equation}
where $Hessf = \nabla^2f$ \cite{180}.
 
The equation \eqref{021} links geometric information about the curvature of the manifold through the Ricci tensor and the geometry of the level sets of the potential function by means of their second fundamental form. For background on Ricci solitons, one can refer to \cite{60,61,62} and references therein.
 
Gradient Ricci solitons have been studied more extensively in last decade.
Petersen et al. \cite{64} and Fern$\acute{a}$ndez-L$\acute{o}$pez et al. \cite{63} studied steady and non-steady gradient Ricci solitons with constant scalar curvature.

Hamilton classified $2$-dimensional shrinking gradient Ricci solitons with bounded curvature \cite{81}. 
In dimension $3$, Ivey proved compact shrinking gradient solitons have constant positive curvature \cite{82}. Noncompact case was classified by Perelman \cite{71}. 
It is shown that if $M^3$ is a shrinking Ricci soliton with bounded curvature, then it has non-negative sectional curvature \cite{62}.
Recently Ni et al. \cite{72} and Naber \cite{74} have given an alternative approach for $3$-dimensional shrinking Ricci solitons. Four dimensional complete noncompact shrinking gradient Ricci soliton were studied by Naber \cite{74}. 
   
On a Riemannian manifold $(M, g)$, the Yamabe flow is defined by     
\begin{equation*}
\dfrac{\partial}{\partial t}g_{ij}=-rg_{ij},  
\end{equation*}
where $r$ is the scalar curvature of $M$. R.S. Hamilton \cite{81} introduced Yamabe flow to solve the Yamabe conjecture, saying that any metric is conformally equivalent to a metric with constant scalar curvature. The Yamabe solitons are special solutions of the Yamabe flow.
A Riemannian manifold $(M, g)$ is said to be a Yamabe soliton if $g$ satisfies the equation
\begin{equation}\label{9.9}  
\mathcal{L}_X g = (\lambda - r)g,   
\end{equation}
where $X$ is a smooth vector field on $M$, $\mathcal{L}_X$ is the Lie derivative along $X$ and $\lambda$ is a real constant. Similarly, a Yamabe soliton is said to be shrinking, steady or expanding if $\lambda < 0$, $\lambda = 0$ or $\lambda > 0$, respectively. If the vector field $X$ is the gradient of a potential function $f$, then $(M, g)$ is said to be a gradient Yamabe soliton and equation \eqref{9.9} reduces to 
\begin{equation}\label{9.7} 
g(\nabla_X\nabla f, Y) = (\lambda -r)g(X, Y).
\end{equation} 
If the potential function $f$ be non-constant, then $M$ is called a nontrivial gradient Yamabe soliton.

Symmetric spaces play a prominent role in differential geometry. Cartan initiated the study of Riemannian symmetric spaces and he introduced the notice of locally symmetric space, that is a Riemannian manifold for which the Riemannian curvature tensor $R$ is parallel \cite{1}. Levy showed that \cite{m} in these spaces, the sectional curvature of every plane remains invariant under parallel transport of the plane along any curve. Ferus started the study of their extrinsic analog, called symmetric or parallel submanifolds and classified all such submanifolds in Euclidean spaces \cite{206}. Semi-symmetric spaces, as a direct generalization of locally symmetric spaces, are the Riemannian manifolds that satisfy the condition $R(X,Y).R=0$, where $X, Y \in \mathfrak{X}(M)$ and $R(X, Y)$ acts as a derivation on $R$. Haesen et al. proved that in these spaces, the sectional curvature of every plane is invariant under parallel transport around any infinitesimal coordinate parallelogram \cite{l}. The classification of semi-symmetric manifolds was described by Szab$\acute{o}$ \cite{27,13}. Thereafter, Deprez studied their extrinsic analog, called semi-symmetric or semi-parallel submanifolds and he obtained the classification of these hypersurfaces and surfaces in Euclidean spaces \cite{201}.

Deszcz generalized the concept of semi-symmetry and introduced pseudo-symmetric manifolds during the study of totally umbilical submanifolds of semi-symmetric manifolds \cite{00,6.0,6.00}. For more details and examples of pseudo-symmetric manifolds, one can refer to \cite{b,6}.
$3$-dimensional pseudo-symmetric manifolds of constant type have been studied by several authors beginning by Kowalski and  Sekizawa \cite{f,g,9,i}. Hashimoto et al. \cite{e} and Calvaruso \cite{d} classified conformally flat pseudo-symmetric spaces of constant type for different dimensions. Cho et al. studied pseudo-symmetric contact homogeneous, quasi-Sasakians, generalized $(\kappa, \mu)$-spaces and trans-Sasakians $3$-manifolds \cite{c}. Three-dimensional semi-symmetric gradient Ricci soliton manifolds were studied by Cho and  Park \cite{30}. They proved that every $3$-dimensional semi-symmetric manifold $M$ admitting a gradient Ricci soliton is locally isometric to either $\mathbb{R}^3$, $\mathbb{S}^3$, $\mathbb{H}^3$, $\mathbb{R}\times\mathbb{S}^2$ or $\mathbb{R}\times\mathbb{H}^2$. 
  
In this paper, we study $3$-dimensional pseudo-symmetric manifolds that the metrics on them are special solitons and we classify these manifolds by proving the following Theorems.
\begin{theorem}\label{th1}  
Let $3$-dimensional Riemannian manifold $(M, g)$ is a gradient Ricci soliton. If $(M, g)$ is pseudo-symmetric then it is locally isometric to 
either $\mathbb{R}^3$, $\mathbb{S}^3$, $\mathbb{H}^3$, $\mathbb{R}\times\mathbb{S}^2$ or $\mathbb{R}\times\mathbb{H}^2$. 
\end{theorem}     
\begin{theorem}\label{th2}
Every $3$-dimensional nontrivial gradient Yamabe soliton pseudo-symmetric manifold is locally isometric to either $\mathbb{R}^3$, $\mathbb{S}^3$, $\mathbb{H}^3$, $\mathbb{R}\times\mathbb{S}^2$ or $\mathbb{R}\times\mathbb{H}^2$. 
\end{theorem}  
\section{Preliminaries}  
Let $(M, g)$ is a $3$-dimensional Riemannian manifold and $\nabla, \mathcal{R}$ and $S$ denote the Levi-Civita connection, the curvature operator and the Ricci operator of $M$, respectively. 

For a symmetric $(0, 2)$-tensor field $g$ on $M$ and $X, Y\in \mathfrak{X}(M)$, we define the endomorphisms $X\wedge_g Y$ and $R(X,Y)$ of $\mathfrak{X}(M)$ by
\begin{equation}\label{7.6} 
(X \wedge_g Y ) Z = g( Y, Z) X - g( X, Z) Y,~~~~~Z\in \mathfrak{X}(M), 
\end{equation}
\begin{equation*}
R(X,Y)=[\nabla_X, \nabla_Y]-\nabla_{[X, Y]}.
\end{equation*}
For a $(0, k)$-tensor field $T$, $k\geq 1$ and a $(0, 4)$-tensor $\mathcal{R}$, the $(0, k + 2)$ tensor fields $\mathcal{R} .T$ and $Q(g, T)$ are defined by \cite{0,6}
\begin{equation}\label{02} 
\begin{array}{ll} 
(\mathcal{R}.T)(X_1,...,X_k; X, Y) & =(R(X,Y).T)(X_1,...,X_k) \\
 & =-T(R(X,Y)X_1, X_2, ...,X_k)\\
 &-...-T(X_1, ...,X_{k-1},R(X,Y)X_k),
\end{array}
\end{equation}     
and  
\begin{equation}\label{3}
 \begin{array}{ll}
Q(g,T)(X_1, ...,X_k; X, Y)&=((X\wedge_g Y).T)(X_1, ..., X_k)\\
&=-T((X \wedge_g Y)X_1, X_2, ..., X_k)\\ 
&-... -T(X_1,...,X_{k-1},(X \wedge_g Y)X_k),
\end{array}
\end{equation} 
respectively. A Riemannian manifold $M$ is said to be pseudo-symmetric if the tensors $R.R$ and $Q(g, R)$ are linearly dependent at every point of $M$, i.e.,
\begin{equation}\label{4}
R.R = L Q(g, R),
\end{equation}   
where $L$ is a smooth function on  the set $U = \lbrace x\in M: R-\dfrac{r}{n(n-1)}G\neq 0 ~at~ x \rbrace $ and the $(0, 4)$-tensor $G$ is given by
\begin{equation*}
G(X_1, X_2, X, Y)= g((X_1\wedge X_2)X, Y)
\end{equation*}
\cite{6}.
This is equivalent to
\begin{equation}\label{5}
(R(X, Y ).R)(U, V, W) = L [((X \wedge_g Y ).R)(U, V, W)],
\end{equation}
holding on the set $U$.
The manifold $M$ is called pseudo-symmetric of constant type if $L$ is a constant. In particular, when $L=0$, $M$ is a semi-symmetric manifold.
\section{Pseudo-symmetric gradient Ricci solitons} 
Let $M^3$ be a pseudo-symmetric manifold which is also a gradient Ricci soliton with a potential function $f$. Let $\lbrace e_i \rbrace $ is a local orthonormal frame field on $M^3$ and
\begin{equation*} 
V=\nabla f = f_0e_0 + f_1e_1 + f_2e_2,  
\end{equation*}
where $f_i = g(\nabla f, e_i)$, $i = 0, 1, 2$. 
Since dim$M=3$, for all $X, Y, Z \in \mathfrak{X}(M)$ we have the formula
\begin{equation}\label{4.01}
R(X, Y ) = SX \wedge Y + X \wedge SY -\frac{r}{2}X \wedge Y. 
\end{equation}
\begin{proposition}\label{1}\cite{080,9} 
A Riemannian $3$-manifold of non-constant curvature is a pseudo-symmetric space with $R . R = L Q(g, R)$ if and only if the eigenvalues of the Ricci tensor locally satisfy the following relations (up to numeration):
\begin{equation*}  
 \mu_1 = \mu_2,~~~~~~ \mu_3 = 2L.  
\end{equation*}
\end{proposition}
For a gradient Ricci soliton, the following formulas hold \cite{2,64}: 
\begin{equation}\label{3.1}
g(\nabla_X \nabla f, Y ) + Ric(X, Y ) -\lambda g(X, Y ) = 0, 
\end{equation}
\begin{equation}\label{3.2}  
R(X, Y )\nabla f = (\nabla_Y S)X - (\nabla_X S)Y,  
\end{equation}
\begin{equation}\label{3.3}  
\Delta f = n\lambda - r,  
\end{equation} 
\begin{equation}\label{3.4}  
\nabla_{\nabla f} r=2Ric(\nabla f, \nabla f),  
\end{equation}
\begin{equation}\label{3.5}   
\Delta_f r = 2 \lambda r- 2|Ric|^2, 
\end{equation}
where $\Delta_f :=\Delta - \nabla_{\nabla f}$ denotes the $f$-Laplacian.
\begin{corollary}\label{..5}\cite{69}
Let $(M, g)$ be a gradient Ricci soliton such that, at each point, the Ricci tensor has a nonzero eigenvalue of multiplicity $n-1$, then 
$\widetilde{M}=N^{n-1}\times \mathbb{R}$. Moreover, if $n\geq 4$ then $N$ is Einstein.
\end{corollary} 
 \begin{proof}[\normalfont{\textbf{Proof of Theorem \ref{th1}}}]
We are going to prove that $3$-dimensional pseudo-symmetric gradient Ricci solitons are semi-symmetric which is equivalent to $L\equiv 0$ on $M$ and then the result follows from \cite{30}. Assume for a contradiction there is $q\in M$ which $L(q)\neq0$. This holds on a small neighborhood $U$ containing $q$. Since $g$ can not be of constant curvature on a small open subset of $U$, then proposition \ref{1} implies the eigenvalues of the Ricci operator $S$ at $q\in M$ be $\mu_1= \mu_2:=\mu$ and $\mu_3=2L$.

If, on a small open subset of $U$, $\mu=2L$ and $(\mu, \mu, \mu)$, $\mu \neq 0$ or $(0, 0, 0)$ are the eigenvalues of the Ricci operator then the restriction of $g$ to $U$ is an Einstein space and $\mu$ is constant on $U$ and on $M$. Therefore, $M^3$ has constant curvature and this is a contradiction.   

Now let $\mu\neq 2L$ and $(\mu, \mu, 2L)$ are the eigenvalues of the Ricci operator. Then there is a local orthonormal frame field $\lbrace e_i\rbrace_{i=0}^2$ around $q$ such that $Se_0=2L e_0$, $Se_i=\mu e_i,~i=1,2$. 
Let   
\begin{equation*}
\nabla_{e_i} e_j = \sum_k B_{ijk} e_k ~~for~~ i, j, k = 0, 1, 2 .
\end{equation*}
Direct calculation implies
\begin{equation}\label{17}
r=2(\mu+L),
\end{equation}
\begin{equation}\label{2} 
\begin{array}{ll}
R(e_1, e_2)= (\mu- L)e_1\wedge e_2,\\  
R(e_1, e_0)=L e_1\wedge e_0,\\
R(e_2, e_0)=L e_2\wedge e_0, 
\end{array}
\end{equation}
\begin{equation}\label{21}
\begin{array}{ll}
Ric(e_i, e_j)=0~ for~ i\neq j,\\ 
Ric(e_1,e_1)=Ric(e_2, e_2)=\mu,\\
Ric(e_0,e_0)=2L,
\end{array}
\end{equation}
\begin{equation*} 
\begin{array}{ll}
(\nabla_{e_0}R)(e_1, e_2)=e_0(\mu-L)e_1\wedge e_2+(\mu-2L)B_{010}e_0\wedge e_2+(\mu-2L)B_{020}e_1\wedge e_0,\\  
(\nabla_{e_1}R)(e_2, e_0)=e_1(L)e_2\wedge e_0+(2L-\mu)B_{101}e_2\wedge e_1,\\
(\nabla_{e_2}R)(e_0, e_1)=e_2(L)e_0\wedge e_1+(2L-\mu)B_{202}e_2\wedge e_1.
\end{array}
\end{equation*} 
The second Bianchi identity yields
\begin{equation}\label{50} 
\begin{array}{ll}
(\nabla_{e_0}R)(e_1, e_2)+(\nabla_{e_1}R)(e_2, e_0)+(\nabla_{e_2}R)(e_0, e_1)=\\
(e_0(\mu-L)+(\mu-2L)(B_{101}+B_{202}))e_1\wedge e_2+(e_1(L)+(2L-\mu) B_{010})e_0\wedge e_1\\
+(e_2(L)+(2L-\mu)B_{020})e_2\wedge e_0 =0.  
\end{array}
\end{equation}
Then
\begin{equation}\label{51}
\begin{array}{ll}
e_0(\mu-L)=(2L-\mu)(B_{101}+B_{202}),\\  
e_1(L)=(\mu-2L) B_{010},\\ 
e_2(L)=(\mu-2L)B_{020}.
\end{array}
\end{equation}
Putting $X=e_0, Y=e_1$ in \eqref{3.2} and using \eqref{51} give 
\begin{equation}\label{52}  
\begin{array}{ll}
f_1L=(\mu-2L)B_{010},\\
e_0(\mu)=(2L-\mu)B_{101}+Lf_0,\\
B_{102}=0,\\
e_1(L)=f_1L.
\end{array}
\end{equation} 
Similarly, for $X=e_i, Y=e_2$, $i=0,1$ we get
\begin{equation}\label{53}   
\begin{array}{ll} 
f_2L=(\mu-2L)B_{020},\\
e_0(\mu)=(2L-\mu)B_{202}+Lf_0,\\
B_{201}=0,\\
e_2(L)=f_2L,
\end{array}
\end{equation}
and
\begin{equation}\label{54}  
\begin{array}{ll} 
e_1(\mu)=f_1(\mu-L),\\
e_2(\mu)=f_2(\mu-L), 
\end{array}
\end{equation}
respectively. Comparing \eqref{52} and \eqref{53}, we deduce
\begin{equation}\label{560}    
B_{101}=B_{202}.
\end{equation}
On the other hand, putting $X=e_i, i=0,1$ and $Y=e_j, j=0,1,2$ in \eqref{3.1}, gives
\begin{equation}\label{56}   
e_0(f_0)+f_1B_{010}+f_2B_{020}+2L-\lambda=0,
\end{equation}
\begin{equation}\label{57}   
e_1(f_1)+f_0B_{101}+f_2B_{121}+\mu-\lambda=0, 
\end{equation}
\begin{equation}\label{59}    
e_0(f_1)+f_0B_{001}+f_2B_{021}=0,
\end{equation}
\begin{equation}\label{60}   
e_0(f_2)+f_0B_{002}+f_1B_{012}=0. 
\end{equation}
Also from \eqref{3.3} one can easily get   
\begin{equation}\label{55}   
div(f_0e_0)+div(f_1e_1)+div(f_2e_2)=3\lambda-2(\mu+L).  
\end{equation}
\begin{lemma}\label{l1} 
The smooth function $L$ on $M$ is a non-constant function.
\end{lemma}  
\begin{proof}
Assume that $L$ is a constant function. Thus from \eqref{51}, \eqref{52} and \eqref{53}, it follows that
\begin{equation}\label{20}
\begin{array}{ll} 
B_{010}= B_{020}=0,\\
f_1=f_2=0.
\end{array}
\end{equation}
Then \eqref{54}, \eqref{56}, \eqref{57} and \eqref{55} reduce to the following equations:
\begin{equation}\label{101}  
e_1(\mu)=e_2(\mu)=0,
\end{equation}
\begin{equation}\label{36}   
e_0(f_0)+2L=\lambda,
\end{equation}
\begin{equation}\label{37}   
f_0 B_{101}+\mu=\lambda,
\end{equation}
\begin{equation}\label{38}   
div(f_0e_0)=3\lambda-2(\mu+L).
\end{equation}
We note that $f_0\neq0$, because if $f_0=0$, \eqref{36} and \eqref{37} imply $\mu=2L$ and this is a contradiction. 
Equation \eqref{3.4} together with \eqref{20} give 
\begin{equation}\label{.25}  
e_0(\mu)=2f_0L.
\end{equation}
In view of \eqref{101}, \eqref{38} and \eqref{.25}, it follows from \eqref{3.5} that
\begin{equation}\label{39}   
f_0^2L=(\lambda-\mu)(2L-\mu).
\end{equation}
Taking derivative of \eqref{39} with respect to $e_0$ and using \eqref{36} and \eqref{.25}, we obtain 
$\mu$ is constant which implies by \eqref{.25}  $Lf_0=0$ and this is a contradiction. Then $L$ is a non-constant function.
\end{proof}
\begin{lemma}
The eigenvalue $\mu$ of the Ricci operator $S$ can not be zero.
\end{lemma}
\begin{proof} 
If $\mu=0$ from equation \eqref{54} we have $f_1=f_2=0$. Then \eqref{56}, \eqref{57} and \eqref{55} reduce
\begin{equation}\label{6..6}   
e_0(f_0)=\lambda-2L,    
\end{equation}
\begin{equation}\label{6..7}   
f_0B_{101}=\lambda-\mu,  
\end{equation}
\begin{equation}\label{68}   
f_0div e_0=2\lambda,   
\end{equation}
respectively. Also \eqref{51}, \eqref{52}, \eqref{53} and \eqref{560} imply
\begin{equation}\label{70} 
\begin{array}{ll} 
e_1(L)=e_2(L)=0,\\
e_0(L)=2Lf_0.
\end{array}
\end{equation}
By means of \eqref{6..6}, \eqref{6..7} and $\mu\neq 2L$, we get $f_0\neq0$.
Using these results in \eqref{3.5}, it follows that $f_0^2+2\lambda=0$. Taking derivative of last equation with respect to $e_0$ and using \eqref{6..6} give $\lambda=2L$. Then $L$ is constant which is impossible by lemma \ref{l1}.
\end{proof}
Then at each point, the Ricci tensor has nonzero eigenvalue of multiplicity $2$ and according corollary \ref{..5}, $M^3\simeq N^{2}\times \mathbb{R}$. Hence $S$ has at least one zero eigenvalue and since $\mu\neq0$, we get $L=0$ and this is a contradiction. Then this case can not occur and this complete the proof.
\end{proof}
\section{Pseudo-symmetry gradient Yamabe solitons} 
Let $(M, g)$ be a $3$-dimensional semi-symmetric manifold and $(\mu_1, \mu_2, \mu_3)$ are the eigenvalues of the Ricci operator at $q \in M$.
It is shown that \cite{40} the condition $R.R=0$ is equivalent to
\begin{equation*}\label{5}
(\mu_i -\mu_j)(2(\mu_i + \mu_j) - r) = 0,
\end{equation*}
and one can consider only the three cases $(\mu, \mu, \mu), (\mu, \mu, 0)$ and $(0, 0, 0)$ at each point, where $\mu$ is a differentiable function on $M$.

The following formulas hold on a gradient Yamabe soliton \cite{41}: 
\begin{equation}\label{6.5}
-Ric(\nabla f, X) = (n - 1)\nabla_X r,
\end{equation} 
\begin{equation}\label{..1}  
\nabla_i G=2r\nabla_i f,
\end{equation}
where $G:=|\nabla f|^2$.
\begin{proposition}\label{pro2}
 Every $3$-dimensional semi-symmetric nontrivial gradient Yamabe soliton manifold is locally isometric to either $\mathbb{R}^3$, $\mathbb{S}^3$, $\mathbb{H}^3$, $\mathbb{R}\times\mathbb{S}^2$ or $\mathbb{R}\times\mathbb{H}^2$.
\end{proposition}
\begin{proof}
If there is an open neighborhood $U$ of $q$ with $(\mu, \mu, \mu), \mu \neq 0$ (or $(0, 0, 0)$,
respectively), then we can see that $M$ is Einstein. Hence $M$ is of constant curvature and is locally isometric to either $\mathbb{R}^3$, $\mathbb{S}^3$ or $\mathbb{H}^3$.

Now we suppose $(\mu, \mu, 0), \mu \neq 0$ are the eigenvalues of the Ricci operator on a small open subset of $M$. Let $\lbrace e_0, e_1, e_2 \rbrace$ be a local orthonormal frame field around $q$ such that 
\begin{equation}\label{0.0} 
\begin{array}{ll}
Ric(e_1, e_1) = Ric(e_2, e_2) = \mu\\ 
Ric(e_1, e_2) = Ric(e_0, e_j) = 0,~~j=0,1,2,  
\end{array} 
\end{equation} 
\begin{equation}\label{0.01}
\begin{array}{ll} 
R(e_1, e_2) = \mu e_1 \wedge e_2\\
all ~others ~R(e_i, e_j) =0. 
\end{array}
\end{equation} 
From \eqref{51} for $L=0$ we have
\begin{equation}\label{11} 
B_{001} = B_{002} = 0,
\end{equation}  
\begin{equation}\label{12}  
e_0(\mu) +\mu(B_{101} + B_{202}) = 0. 
\end{equation} 
Let $M$ admits a nontrivial gradient Yamabe soliton with a potential vector field
$\nabla f= f_0e_0 + f_1e_1 + f_2e_2$, where $f_i = g(\nabla f, e_i), i = 0, 1, 2$. Suppose
\begin{equation*}
\nabla_{e_i}e_j =\Sigma_k B_{ijk}e_k~~ for~ i, j, k = 0, 1, 2.
\end{equation*}  
The tangent space $T_qM$ at $q$ is $T_qM = D_0(q) \oplus D_1(q)$, where $D_0(q) = span_{\mathbb{R}}\lbrace e_0\rbrace$ and $D_1(q) = span_{\mathbb{R}}\lbrace e_1, e_2\rbrace$.
Putting $X=e_i, ~Y=e_j,~i,j=0,1,2$ in \eqref{9.7} and using \eqref{11} give 
\begin{equation}\label{3.02}
e_1(f_1)+f_0B_{101}+f_2B_{121}=\lambda-2\mu, 
\end{equation}
\begin{equation}\label{3.333}
e_2(f_2)+f_0B_{202}+f_1B_{212}=\lambda-2\mu,
\end{equation}
\begin{equation}\label{3.444}
e_0(f_0)=\lambda-2\mu,
\end{equation} 
\begin{equation}\label{3.555}
e_0(f_1)+f_2B_{021}=0,
\end{equation}
\begin{equation}\label{3.6} 
e_0(f_2)+f_1B_{012}=0,
\end{equation}
\begin{equation}\label{3.8}
e_1(f_2)+f_1B_{112}+f_0B_{102}=0, 
\end{equation}
\begin{equation}\label{3.10}
e_2(f_1)+f_0B_{201}+f_2B_{221}=0.
\end{equation}
Replacing $X$ by $e_i$, $i=0, 1, 2$ in \eqref{6.5} and using \eqref{0.0} we get  
\begin{equation}\label{13} 
e_0(\mu) = 0, ~~~~~4e_1(\mu)=-f_1\mu,~~~~~~~4e_2(\mu)=-f_2\mu.
\end{equation}
In view of \eqref{12}, \eqref{13} and $\mu\neq 0$, it follows that
\begin{equation}\label{3.11} 
B_{101} =- B_{202}.
\end{equation}
Equation \eqref{..1} for $i=0$, together with \eqref{3.444}, \eqref{3.555} and \eqref{3.6} yields
\begin{equation}\label{3.12} 
f_0(\lambda-4\mu)=0.
\end{equation} 
Then we have the following cases.
\begin{itemize} 
\item[case 1.] $f_0=0$ and $\lambda\neq4\mu$\\
In this case, from equation \eqref{3.444} $\mu$ is constant and then \eqref{13} implies $f_1=f_2=0$. Hence $f$ is constant and $(M, g)$ is a trivial gradient Yamabe soliton, contradicting our initial assumption.
\item[case 2.] $f_0\neq0$ and $\lambda=4\mu$\\
Since $\mu$ is constant, equation \eqref{13} gives $f_1=f_2=0$. Hence \eqref{3.8} and \eqref{3.10} yield
\begin{equation}\label{3.13} 
B_{102}=B_{120}=B_{201}=B_{210}=0.
\end{equation}
Also from \eqref{3.02}, \eqref{3.333} and \eqref{3.11} we have  
\begin{equation}\label{3.14} 
B_{101}=B_{110}=B_{202}=B_{220}=0.
\end{equation}
Equations \eqref{11}, \eqref{3.13} and \eqref{3.14} imply that each conullity distribution $D_1$ and $D_0$ is integrable and forms a totally geodesic submanifold. Then $M^3$ is a local product space of a $2$-dimensional manifold $N_1^2$ and $1$-dimensional $\mathbb{R}$ almost everywhere.
Since $\mu$ is constant $N_1^2$ is locally isometric to either $\mathbb{S}^2$, $\mathbb{H}^2$ or $\mathbb{R}^2$ and then $M^3$ is locally isometric to either $\mathbb{R}\times\mathbb{S}^2$, $\mathbb{R}\times\mathbb{H}^2$ or $\mathbb{R}^3$.
\item[case 3.] $f_0=0$ and $\lambda=4\mu$\\
This case can not occur, because from \eqref{3.444} we get $\lambda=2\mu$. Hence $\mu=0$, which is a contradiction.
\end{itemize}  
Now we show that the combination of the three cases $(\mu, \mu, \mu), (\mu, \mu, 0)$ and $(0, 0, 0)$ can not occur in an open neighborhood of $q$. Let $M=F_0\cup F_1\cup F_2$ such that 
$$F_0=\lbrace x\in M: {\mbox{the eigenvalues of the Ricci opertor $S$ at $x$ is }} (\mu,\mu,\mu), \mu\neq0\rbrace,$$ 
$$F_1=\lbrace x\in M: {\mbox{the eigenvalues of the Ricci opertor $S$ at $x$ is }} (\mu,\mu,0), \mu\neq0\rbrace,$$
$$F_2=\lbrace x\in M: {\mbox{the eigenvalues of the Ricci opertor $S$ at $x$ is }} (0,0,0) \rbrace.$$
Then $F_0^\circ\cup F_1^\circ\cup F_2^\circ$ is an open dense subset in $M$. We define a continuous function $m$ on $M$ which is
$$\begin{array}{ll}
m: & M\longrightarrow \mathbb{R}_+  \\  
& x\longrightarrow \max_{i=0,1,2}(|\mu_i(x)|).  
 \end{array}$$
\noindent \newline
 Then the scalar curvature on $F_i,  i=0,1,2$, is $|r(x)|=n(x)m(x)$ for $n(x)$ in $\lbrace 0, 2, 3\rbrace$.
Since on each $F_i$, $m$ is constant, we deduce that $m$ is constant on $M$ (a continuous function taking a finite number of values on a connected space is constant). If $m\equiv0$ then $F_2=M$. Otherwise $n(x)=\dfrac{|r(x)|}{m(x)}$  is continuous and takes values in $\lbrace 2, 3\rbrace$. Then $n(x)$ is constant. Hence either $M=F_0$, $M=F_1$ or $M=F_2$.
\end{proof}   
\begin{proposition}\label{pro3}
Let $(M, g)$ be a $3$-dimensional Riemannian manifold and the metric on it is a nontrivial gradient Yamabe soliton. Then $M$ is  pseudo-symmetric if and only if it is semi-symmetric.
\end{proposition}
\begin{proof}
Since every semi-symmetric manifold is pseudo-symmetric, it is enough to proof the ``only if '' part which is equivalent to $L\equiv 0$ on $M$. Assume for a contradiction there is $q\in M$ which $L(q)\neq0$ and argue like in the proof of Theorem \ref{th1}. Since $g$ can not be of constant curvature on a small open subset of the neighborhood $U$ containing $q$, then proposition \ref{1} implies the eigenvalues of the Ricci operator $S$ at $q\in M$ be $\mu_1= \mu_2:=\mu$ and $\mu_3=2L$. If $\mu=2L$, then the restriction of $g$ to $U$ is an Einstein space and $\mu$ is constant on $U$ and on $M$. Therefore, $M^3$ has constant curvature and this is a contradiction.   

Now let $\mu\neq 2L$ and $Se_0=2L e_0$,  $Se_1=\mu e_1$ and $Se_2=\mu e_2$. Putting $X=e_i,~Y=e_j$, $i,j=0,1,2$ in \eqref{9.7} and using \eqref{17} give 
\begin{equation}\label{5.2}
e_1(f_1)+f_0B_{101}+f_2B_{121}=\lambda-2(\mu+L), 
\end{equation}
\begin{equation}\label{5.3}
e_2(f_2)+f_0B_{202}+f_1B_{212}=\lambda-2(\mu+L),
\end{equation} 
\begin{equation}\label{5.4}
e_0(f_0)+f_1B_{010}+f_2B_{020}=\lambda-2(\mu+L),
\end{equation} 
\begin{equation}\label{5.5} 
e_0(f_1)+f_0B_{001}+ f_2B_{021}=0,
\end{equation}
\begin{equation}\label{5.6}
e_0(f_2)+f_1B_{012}+f_0B_{002}=0,
\end{equation}
\begin{equation}\label{5.7}
e_1(f_0)+f_1B_{110}+f_2B_{120}=0, 
\end{equation}
\begin{equation}\label{5.8}
e_1(f_2)+f_1B_{112}+f_0B_{102}=0, 
\end{equation}
\begin{equation}\label{5.9}
e_2(f_0)+f_1B_{210}+f_2B_{220}=0,
\end{equation}
\begin{equation}\label{5.10}
e_2(f_1)+f_0B_{201}+f_2B_{221}=0.
\end{equation}
By virtue of \eqref{17} and \eqref{21} in \eqref{6.5} for $X=e_i$, $i=0, 1, 2$ we have
\begin{equation}\label{5.11} 
2e_0(\mu+L) = -f_0L, ~~~~~4e_1(\mu+L)=-f_1\mu,~~~~~~~4e_2(\mu+L)=-f_2\mu.
\end{equation}
Applying above equations in \eqref{..1} yields
\begin{equation}\label{5.12} 
f_i(\lambda-4(\mu+L))=0 ,~~i=0,1,2.
\end{equation}
Since $f$ is a non-constant function, then $\lambda=4(\mu+L)$. Taking derivative with respect to $e_i,~i=0,1,2$ and using \eqref{5.11}, we obtain 
\begin{equation}\label{5.14}   
f_0=0,~~f_1\mu=0,~~f_2\mu=0. 
\end{equation} 
But $\mu\neq 0$, because otherwise $L=\dfrac{\lambda}{4}$ is constant and from \eqref{51} we have
 \begin{equation}\label{5.141}   
B_{020}=B_{010}=0. 
\end{equation}  
On the other hand, in view of \eqref{5.4}, \eqref{5.14} and \eqref{5.141}, it follows that $L=\dfrac{\lambda}{2}$. Then $L=0$ and this is a contradiction. Hence $f$ is constant, contradicting our assumption and then this case can not occur.
\end{proof}   
\begin{proof}[\normalfont{\textbf{Proof of Theorem \ref{th2}}}]
It follows from proposition \ref{pro2} and proposition \ref{pro3}.
\end{proof}
\section*{Acknowledgement} 
The first author would like to express her gratitude to Prof. Thomas Delzant and Department of Mathematics, Strasbourg University, that provided all facilities for studying during her research period in that university. She also would like to thank Prof. Charles Frances for his valuable comments in order to improve the paper.


\end{document}